\documentclass[letter,11pt,reqno]{amsart}

\usepackage[utf8]{inputenc}

\usepackage{etex}

\usepackage{xcolor}

\definecolor{verydarkblue}{rgb}{0,0,0.5}

\usepackage[
    breaklinks,
    colorlinks,
    citecolor=verydarkblue,
    linkcolor=verydarkblue,
    urlcolor=verydarkblue,
    pagebackref=true,
    hyperindex
]{hyperref}

\backrefenglish

\usepackage{fancyhdr}

\usepackage[
    hscale=0.70,
    vscale=0.775,
    headheight=13pt,
    centering,
]{geometry}

\usepackage{amsmath}
\usepackage{amsthm}
\usepackage{amssymb}
\usepackage{mathtools}
\usepackage{stmaryrd}
\usepackage{mathdots}
\usepackage{framed}
\usepackage[capitalize]{cleveref}
\usepackage{array}
\usepackage[alphabetic]{amsrefs}
\usepackage[all,cmtip]{xy}

\theoremstyle{plain}

\newtheorem{introtheorem}{Theorem}

\crefname{introtheorem}{Theorem}{Theorems}

\newtheorem{theorem}{Theorem}[section]
\newtheorem{proposition}[theorem]{Proposition}
\newtheorem{lemma}[theorem]{Lemma}
\newtheorem{corollary}[theorem]{Corollary}

\theoremstyle{definition}

\newtheorem{definition}[theorem]{Definition}

\newtheorem{remark}[theorem]{Remark}

\numberwithin{figure}{section}

\numberwithin{equation}{section}

\IfFileExists{./article-style.tex}{\input{article-style.tex}}{}

\usepackage{marginnote}

\def\R{{\mathbb R}}
\def\C{{\mathbb C}}

\def\P{{\mathbb P}}
\def\F{{\mathbb F}}

\def\cC{\mathcal{C}}

\def\cF{\mathcal{F}}

\def\cQ{\mathcal{Q}}

\def\cS{\mathcal{S}}

\def\F{\mathcal{F}}

\def\M{\mathcal{M}}
\def\O{\mathcal{O}}
\def\U{\mathcal{U}}
\def\M{\mathcal{M}}

\def\a{\alpha}

\def\f{\phi}

\def\ep{\epsilon}

\def\n{\nu}
\def\m{\mu}

\def\p{\pi}
\def\r{\rho}

\def\t{\tau}

\def\D{\Delta}
\def\G{\Gamma}

\def\Om{\Omega}

\def\.{\cdot}
\def\^{\widehat}
\def\~{\widetilde}

\def\ov{\overline}

\def\surj{\twoheadrightarrow}
\def\inj{\hookrightarrow}

\def\de{\partial}

\def\({\left(}
\def\){\right)}

\renewcommand{\and}{ \ \ \text{ and } \ \ }

\def\reg{\mathrm{reg}}
\def\sing{\mathrm{sing}}

\def\an{\mathrm{an}}

\def\KR{\mathrm{KR}}

\DeclareMathOperator{\coker} {coker}

\DeclareMathOperator{\Gr} {Gr}

\DeclareMathOperator{\Pic} {Pic}

\DeclareMathOperator{\Sing} {Sing}

\title{Smooth solutions to the complex Plateau problem}

\author{Tommaso de Fernex}

\address{Department of Mathematics, University of Utah, Salt Lake City, UT 48112, USA}
\email{{\tt defernex@math.utah.edu}}


\thanks{Research partially supported by NSF Grant DMS-1700769}

\begin{document}

\begin{abstract}
Building on work of Du, Gao, and Yau, we give a characterization of smooth solutions, up to normalization, of 
the complex Plateau problem for strongly pseudoconvex Calabi--Yau CR manifolds of dimension $2n-1 \ge 5$
and in the hypersurface case when $n=2$. The latter case 
was completely solved by Yau for $n \ge 3$ but only partially solved by Du and Yau for $n=2$. 
As an application, we determine the existence of a link-theoretic invariant of 
normal isolated singularities that distinguishes smooth points
from singular ones.
\end{abstract}

\maketitle

\section{Introduction}

The complex Plateau problem has been studied in the context of CR manifolds
in the articles of Rossi \cite{Ros65} and Harvey and Lawson \cite{HL75} (cf.\ \cite{LY98,HL00})
where it is proved that
a strongly pseudoconvex compact CR manifold $M$ of dimension $2n-1 \ge 3$
that is contained in the boundary of a strongly pseudoconvex bounded domain of $\C^N$
is the boundary of a unique complex analytic variety $V \subset \C^N \setminus M$ 
with boundary regularity and isolated singularities.%
\footnote{The requirement that $M$ is contained in the boundary of a strongly pseudoconvex bounded domain of $\C^N$ 
is needed to ensure that the closure of $V$ has no self intersections along $M$, and hence guarantee 
boundary regularity; one can work around it using \cite[Theorem~10.4']{HL00}.}
We will refer to $V$ as the Stein filling of $M$.
What remain to be determined are intrinsic conditions on $M$ that characterize when $V$
(or its normalization, to be precise) is smooth. 

In the hypersurface case $N = n+1$, this last question was completely answered 
for $n \ge 3$ by Stephen Yau in \cite{Yau81}, 
where the smoothness of the Stein filling $V$ is proved to be characterized by the vanishing of the
Kohn--Rossi cohomology groups $H^{p,q}_\KR(M)$ for $1 \le q \le n-2$.
Yau's theorem answered affirmatively a conjecture of Kohn and Rossi from \cite{KR65}.
The main difficulty to extend the result to the case $n=2$ 
is that the Kohn--Rossi cohomology groups are infinite dimensional in this case.
Over a quarter of a century passed before
the hyperpsurface case was finally addressed in the remaining case $n=2$, first in \cite{LY07} still
in terms of the Kohn--Rossi cohomology of $M$, and then in \cite{DY12}
using a different numerical invariant of $M$ denoted by $g^{(1,1)}(M)$. 
The main result of \cite{DY12} states that, under the condition that 
$M$ has holomorphic De Rham cohomology $H_h^2(M) = 0$, the 
Stein filling is smooth if and only if $g^{(1,1)}(M) = 0$. 
This invariant has also been studied in the papers \cite{DLY11,DG14}.
Yau communicated to us that the definition of this invariant came as a result of a study started in the papers
\cite{LYY01,Yau04} where Bergman functions are used to study isolated singularities.

In order to extend the results beyond the hypersurface case, 
a new invariant of $M$ which generalizes $g^{(1,1)}(M)$ to higher dimensions 
is introduced in \cite{DGY16}. The definition of this invariant uses
holomorphic cross sections of exterior powers of the complex bundle $\^T(M)^*$. 
Here $\^T(M)^*$ is the dual of the holomorphic tangent bundle
$\^T(M) = \C T(M)/T_{0,1}(M)$
of $M$, with $T_{0,1}(M)$ denoting the conjugate of the CR structure of $M$;
the notion of holomorphic cross section is defined using the tangential
Cauchy--Riemann condition on $M$. 
Denoting by $\G_h(M,E)$ the space of global holomorphic cross sections
of a holomorphic bundle $E$ on $M$, Du, Gao and Yau define
\[
g^{(\wedge^p1)}(M) := \dim \coker\big(\wedge^p\G_h(M,\^T(M)^*) \to \G_h(M,\wedge^p\^T(M)^*)\big).
\]
This invariant is zero whenever the normalization of $V$ is smooth, 
and the two main theorems of \cite{DGY16} provide partial converses to this property when $p=n$. 
The first states that if $n \ge 3$ and $g^{(\wedge^n1)}(M) = 0$ 
then the normalization of $V$ has at most rational singularities.
The second proves that if $M$ is a Calabi--Yau CR manifold of dimension 5 (i.e., $n = 3$) 
and $g^{(\wedge^31)}(M) = 0$, then the normalization of $V$ is smooth.
This point of view is further explored in \cite{Du} where an alternative characterization
of smooth solutions of the complex Plateau problem in the $n \ge 3$ hypersurface case is given. 

The purpose of this paper is to show that this approach leads in fact to a characterization of
smooth solutions to the complex
Plateau problem, up to normalization, in the Calabi--Yau case for all $n \ge 3$
and in the hypersurface case for all $n \ge 2$. 

\begin{introtheorem}
\label{th:Plateau-CY}
For any strongly pseudoconvex compact Calabi--Yau CR manifold  $M \subset \C^N$ of dimension $2n-1 \ge 5$
that is contained in the boundary of a strongly pseudoconvex bounded domain of $\C^N$,
the Stein filling $V$ has smooth normalization if and only if $g^{(\wedge^n1)}(M) = 0$.
\end{introtheorem}

\begin{introtheorem}
\label{th:Plateau-hypersurface}
For any strongly pseudoconvex compact CR manifold  $M \subset \C^{n+1}$ of dimension $2n-1 \ge 3$
that is contained in the boundary of a strongly pseudoconvex bounded domain of $\C^{n+1}$,
the Stein filling $V$ is smooth if and only if $g^{(\wedge^n1)}(M) = 0$.
\end{introtheorem}

The proofs of the above theorems rely on known cases of the Zariski--Lipman conjecture \cite{Lip65}.
Were this conjecture known to hold for all normal surface singularities, the proof 
of \cref{th:Plateau-CY} would extend to the case $n=2$. 

Following \cite{DY12,DGY16}, the proofs rely on computing $g^{(\wedge^n1)}(M)$ from a related
numerical invariant of the singularities of $V$. 
It was conjectured by Stephen Yau that, for $n=2$, such invariant is positive for all normal surface singularities (see \cite{DG14}, where this conjecture is solved for rational surface singularities);
the conjecture was extended to higher dimensions 
in \cite{Du}. The results of this paper confirm these conjectures for 
normal Gorenstein singularities of dimension $n \ge 3$ and locally
complete intersection singularities of dimension $n=2$ (cf.\ \cref{th:Plateau-Gor,th:Plateau-lci}).

The case $n \ge 3$ of \cref{th:Plateau-hypersurface} 
was already proved by a different argument in \cite{Du}, and
can be viewed as an alternative solution of the complex Plateau problem compared 
to the one in \cite{Yau81}. Since hypersurface singularities are Gorenstein and
boundaries of Stein spaces with boundary regularity and trivial canonical bundle
are examples of Calabi--Yau CR manifolds, this case can also be seen following from \cref{th:Plateau-CY}.
The original contribution of \cref{th:Plateau-hypersurface} is in the case $n=2$ where the theorem improves
upon the main result of \cite{DY12} by removing the assumption that
$H_h^2(M) = 0$, hence finally completing the characterizarion of smooth solutions of 
the Plateau problem in the hypersurface case for all $n \ge 2$. 
The above results are also closely related to the recent works
\cite{TYZ13,LYZ15,YZ17,YZ}.

The reason why one looks at the problem up to normalization is that the boundary $M$ does not see
the difference between $V$ and its normalization, 
$V$ having only isolated singularities. In fact, it is natural to look at the normalization,
as embeddable strongly pseudoconvex
compact CR manifolds have unique normal Stein filling with boundary regularity and any other Stein
filling is a projection of it.
The issue of taking normalization disappears in the hypersurface
case for the simple reason that 
isolated locally complete intersection singularities of dimension $n \ge 2$ are already normal. 

The complex Plateau problem is closely related to the question asking whether the link of
a normal isolated complex singularity carries enough information to distinguish smooth points
from singular ones. The question traces back at least to the work of
Mumford, Brieskorn, and Milnor \cite{Mum61,Br66a,Br66b,Mil68}
which showed that the diffeomorphism class of the link distinguishes smooth points in dimension 2 but not 
in dimension 3 or higher. The question has recently been revived 
by McLean in \cite{McL16}, where a connection is established between the contact structure of the link
and the minimal log discrepancy of the singularity, an invariant of singularities
coming up in the minimal model program.
In his paper, McLean asks whether the contact structure of the link
suffices to characterize smoothness, and proves that this is the case conditionally to
a conjecture of Shokurov on minimal log discrepancies \cite{Sho02}.
This in particular implies that the property holds in dimension 3 and for locally complete intersection
singularities. 
More cases were later established in \cite{dFT17}.

Here we use CR geometry and \cref{th:Plateau-CY}
to prove the following general result (see \cref{th:CR-link-invariant} for a more precise statement)

\begin{introtheorem}
\label{th:link-invariant}
There exists a link-theoretic analytic invariant of singularity which distinguishes smooth points
among normal isolated singularities.
\end{introtheorem}

The CR structure of a given link of a singularity further enhances the contact structure
but, differently from the latter, it does not define an invariant of the singularity as in general 
it depends on the embedding of the singularity and the radius of the link.
In order to define an invariant of singularity using CR geometry, 
it is necessary to introduce an equivalence
relation among strongly pseudoconvex CR manifolds. 
The notion of cohomologically rigid CR deformation equivalence was introduced in \cite{dFT17} for this purpose. 
The `cohomolocally rigid' part of the term refers to the condition that certain Kohn--Rossi cohomological dimensions
are preserved under deformation, and was included in the definition with the results of \cite{Yau81} in mind. 
The results of \cite{DGY16,Du} and of this paper indicate that such notion
needs to be adjusted. 
We do this in the last section of the paper, where we replace most of the conditions on Kohn--Rossi cohomology
with an analogous condition on the invariant $g^{(\wedge^n 1)}$ defined above
and prove that the resulting equivalence relation defines a link-theoretic invariant of 
normal isolated singularities
which can be used to distinguish smooth points from singular ones.

\subsection*{Acknowledgements}

We thank Roi Docampo for many valuable discussions, especially on the topic of
Nash tranformations, and Hugo Rossi and Stephen Yau for historical comments
and for their interest in our work. 
We wish to express our indebtedness to the papers \cite{DGY16,Du}
which have inspired this work, and to Stephen Yau for bring to our attention
his work on the complex Plateau problem in connection to our previous
paper on links of singularities \cite{dFT17}.
We thank the referee for suggesting relevant references.

\section{Preliminaries}

In this section we briefly recall some basic definitions about CR manifolds; 
general references are \cite{Tan75,DT06}.
All manifolds will be implicitly assumed to be $C^\infty$, connected, and orientable.
We will only consider CR manifolds of CR codimension 1.

A CR structure on a manifold $M$ with real dimension $2n-1$ is an
$(n-1)$-dimensional subbundle $T_{1,0}(M)$
of the complexified tangent bundle $\C T(M) := T(M) \otimes \C$ that is closed under Lie bracket and satisfies
$T_{1,0}(M)\cap T_{0,1}(M) = {0}$, where $T_{0,1}(M) := \ov{T_{1,0}(M)}$.
A manifold $M$ equipped with a CR structure $T_{1,0}(M)$ is called a CR manifold. 
As long as $T_{1,0}(M)$ is clear from the context or implicit in the discussion, 
we simply write $M$ to denote the CR manifold.

The Levi distribution of a CR manifold $M$ is the subbundle
$H(M) = \Re(T_{1,0}(M) \oplus T_{0,1}(M))$ of $T(M)$.
A pseudo-Hermitian structure on $M$ consists of a global differential 1-form $\theta$
on $M$ such that $H(M) = \ker\theta$.
A CR manifold $M$ is said to be strongly pseudoconvex
if for some pseudo-Hermitian structure $\theta$ the Levi form $-i\,d\theta$ is positive definite, 
that is, $-i\,d\theta(Z,\bar Z) > 0$ for every nonvanishing $\cC^\infty$ local section $Z$ of $T_{1,0}(M)$.

Let $M$ be a CR manifold with structure $T_{1,0}(M)$. 
In the following paragraphs, we work locally on $M$ 
and use the symbol $\G$ to denote spaces of $\cC^\infty$ sections
over an unspecified open subset of $M$. 
We follow \cite[Section~1.2]{Tan75}, to which we refer for more details. 

The operator 
\[
\bar\de_b \colon \G(M\times\C) \to \G(T_{0,1}(M){}^*)
\]
defined by setting $\bar\de_b u (\bar Z) := \bar Zu$ for any $Z \in \G(T_{1,0}(M))$
is called the (tangential) Cauchy--Riemann operator, and  
a solution to the equation $\bar\de_b u = 0$ is called a holomorphic function.

A complex vector bundle $E$ over $M$ is said to be holomorphic
if there is a differential operator 
\[
\bar\de_E \colon \G(E) \to \G(E \otimes T_{0,1}(M){}^*)
\]
satisfying the conditions
$\bar Z(fu) = \bar Zf\.u + f\. \bar Zu$ and $[\bar Z,\bar W]u = \bar Z\bar W u - \bar W \bar Z u$
for all $u \in \G(E)$, $f \in \cC^\infty(M)$, $Z,W\in \G(T_{1,0}(M))$, 
where we put $\bar Zu := \bar\de_Eu(\bar Z)$. 
The operator $\bar\de_E$ is called the Cauchy--Riemann operator of $E$, and a solution 
to the equation $\bar\de_E u = 0$ is called a holomorphic cross section of $E$. 

The holomorphic tangent bundle $\^T(M) := \C T(M)/T_{0,1}(M)$
is a complex vector bundle of rank $n$ (where $2n-1$ is the real dimension of $M$)
with holomorphic structure given by the Cauchy--Riemann operator $\bar\de = \bar\de_{\^T(M)}$
defined as follows:
for a given $u \in \G(\^T(M))$, pick $U \in \G(\C T(M))$ mapping to $u$ under
the quotient map $\p \colon \C T(M) \to \^T(M)$, and define $\bar\de u$ by setting
$\bar\de u(\bar Z) := \p([\bar Z,U])$ for every $Z \in \G(T_{1,0}(M))$. 
One checks that this definition is independent of the choice of $U$
and the resulting operator $\bar\de$ satisfies the above conditions.
If $M$ is a real hypersurface on an $n$-dimensional complex manifold $X$ and
$T_{1,0}(M) = T^{1,0}(X) \cap \C T(M)$ is the CR structure induced from the complex structure of $X$ 
(here $T^{1,0}(X)$ is the holomorphic tangent bundle of $X$), 
then we have $\^T(M) = T^{1,0}(X)|_M$.
A compact CR manifold $M$ of dimension $2n-1$ is said to be 
Calabi--Yau if the complex line bundle $\wedge^n\^T(M)^*$, which is called the canonical bundle of $M$,
is trivial as a holomorphic bundle.

The Kohn--Rossi cohomology of a CR manifold $M$ is defined using the bundles $\wedge^k\^T(M)^*$. 
For every $k$ and $p$, denote
by $F^p(\wedge^k\^T(M)^*)$ the subbundle of $\wedge^k\^T(M)^*$ consisting of those local sections
$\a$ which satisfy the equation 
\[
\a(Z_1,\dots,Z_{p-1},\bar W_1,\dots,\bar W_{k-p+1}) = 0
\]
for all $Z_i \in \C T(M)_x$ and $W_j \in T_{1,0}(M)_x$, 
where $x \in M$ is the origin of $\a$. 
The collection $\{F^p(\wedge^k\^T(M)^*)\}$ gives a filtration of the de Rham complex, 
and the Kohn--Rossi cohomology groups $H_\KR^{p,q}(M)$ are computed by the
groups $E_1^{p,q}$ in the first page of the spectral sequence
associated to this filtration.
Equivalently, for every fixed $p$ the Cauchy--Riemann operator of $\wedge^p\^T(M)^*$ induces operators
\[
\G(\wedge^p\^T(M)^* \otimes \wedge^q T_{0,1}(M){}^*) \to 
\G(\wedge^p\^T(M)^* \otimes \wedge^{q+1} T_{0,1}(M){}^*)
\]
for all $q$. These operators fit into a complex
and $H_\KR^{p,q}(M)$ is the $q$-th cohomology group of this complex.

\section{Proofs of \cref{th:Plateau-CY,th:Plateau-hypersurface}}

The proofs combine results from \cite{DGY16} with an idea from \cite{dFT17}. 
Throughout this section, let
$M \subset \C^N$ be a strongly pseudoconvex compact CR manifold of dimension $2n-1 \ge 2$
that is contained in the boundary of a strongly pseudoconvex bounded domain of $\C^N$,  
and let $V \subset \C^N \setminus M$ be 
the Stein filling of $M$ with boundary regularity and isolated singularities. 

Let $U = V \setminus V_\sing$ denote the smooth
locus of $V$, with inclusion $i \colon U \inj V$.
The authors of \cite{DGY16} work with the sheaves
\[
\bar{\bar\Om}_V^{p} := i_*\Om_U^p.
\]
It is proved in \cite[Theorem~A]{Siu70} that these sheaves are coherent.
Here we work directly on the normalization $\m \colon \~V \to V$. After identifying $U$ with 
its image $\m^{-1}(U)$ in $\~V$, 
we denote by $j \colon U \inj \~V$ the inclusion in $\~V$. 
On $\~V$, we consider the sheaves of reflexive differentials
\[
\Om_{\~V}^{[p]} := (\Om_{\~V}^p)^{**}.
\]
Note that $\Om_{\~V}^{[p]} = j_*\Om_U^p$, since $\~V$ is normal.
There is a natural map $\wedge^p\Om_{\~V}^{[1]} \to \Om_{\~V}^{[p]}$
whose cokernel is supported within the singular locus of $\~V$.
Since $\~V$ has isolated singularities, we can define the dimension of the cokernel
as the sum of the dimensions of its non-zero fibers.

\begin{lemma}
\label{lem:g(M)}
We have
\[
g^{(\wedge^p1)}(M) = \dim \coker\big(\wedge^p\Om_{\~V}^{[1]} \to \Om_{\~V}^{[p]}\big).
\]
\end{lemma}

\begin{proof}
This follows directly from \cite[Lemma~4.7]{DGY16}, which states that
\[
g^{(\wedge^p1)}(M) = \dim \coker\big(\wedge^p\bar{\bar\Om}_V^1 \to \bar{\bar\Om}_V^p\big).
\]
The equivalence of the two formulas can be seen using the Leray spectral sequence and the equation
$\bar{\bar\Om}_V^{p} = \m_*\Om_{\~V}^{[p]}$.
Alternatively, one can directly mimic the proof of \cite[Lemma~4.7]{DGY16}.
\end{proof}

\begin{lemma}
\label{lem:CY}
$M$ is Calabi--Yau if and only if $\Om_{\~V}^{[n]} \cong \O_{\~V}$.
\end{lemma}

\begin{proof}
This fact is essentially explained in \cite{DGY16} (see the remark following
\cite[Definition 2.6]{DGY16}).
The proof of \cite[Lemma~4.6]{DGY16} goes by showing that any nonzero holomorphic cross section $\theta$ of 
$\wedge^p\^T(M)^*$ extends holomorphically to a one side neighborhood of $M$ in $U$ and hence to
a global section of $\Om_U^p$. 
Since the natural map $\G(U,\Om_U^p) \to \G(\~V,\Om_{\~V}^{[p]})$ is an isomorphism, 
this section further extends to a global section $\theta'$ of $\Om_{\~V}^{[p]}$. 
Here we take $p=n$.
As $\Om_{\~V}^{[n]}$ is a reflexive sheaf of rank 1 on a normal space, 
the section $\theta' \colon \O_{\~V} \to \Om_{\~V}^{[n]}$
is either an isomorphism or it has cokernel supported on an analytic hypersurface $Z \subset \~V$
which, being positive dimensional, will have a non-empty boundary within $M$. 
If we assume that $\theta$ is nowhere vanishing, then we have $Z = \emptyset$
and hence $\theta'$ gives a trivialization $\Om_{\~V}^{[n]} \cong \O_{\~V}$. 
Conversely, a trivialization $\Om_{\~V}^{[n]}$, viewed as a holomorphic bundle over $\~V$, yields 
a holomorphic trivialization of $\wedge^n\^T(M)^*$.
\end{proof}

The idea we take from \cite{dFT17} is to use Nash transformations
to detect smoothness. In \cite{dFT17} we used the Nash transformation of 
the sheaf of differentials $\Om^1$, 
which is also known as the Nash blow-up of the variety. 
Here, we use the Nash transformation of sheaf of reflexive differentials $\Om^{[1]}$ instead. 
In the first case we relied on a theorem of Nobile \cite{Nob75} stating that the Nash blow-up of 
a complex analytic variety is an isomorphism if and only if the variety is smooth. 
Here we rely on known cases of the Zariski--Lipman conjecture to deduce the analogous property
for the Nash transformation of $\Om^{[1]}$.  

In algebraic geometry, the Nash transformation of
a coherent sheaf $\cF$ of rank $r$ on a Noetherian scheme $S$ 
is the universal projective birational morphism 
$\n \colon N(\cF) \to S$ such that the pull-back sheaf $\n^*\cF$ admits a
locally free quotient of the same rank $r$. 
The construction, which essentially goes back to Grothendieck, uses the 
Grassmannian bundle $\Gr(\cF,r)$ 
parametrizing locally free quotients of $\cF$ of rank $r$.
The details of the construction and basic properties can be found in \cite{OZ91}.
Notice that the Grassmannian bundle $\Gr(\cF,r)$ is defined as a scheme over $S$
via its functor of points \cite[Ch.~I, Th\'eor\`eme~9.7.4]{GD71}. 

Working with coherent sheaves on a complex analytic variety, 
we cannot rely on the language of schemes, but
it is not hard to generalize the construction to the analytic setting.
The argument is standard, and is included here for the convenience of the reader.

Let $\cF$ be a coherent sheaf of rank $r$ on a complex analytic variety $X$. 
In order to construct the Grassmannian bundle of $\cF$, 
we first assume that $\cF$ is generated by finitely many global holomorphic sections. 
We fix a set of generators $s_1,\dots,s_m$
of $\cF$ and look at the surjection
\[
\f \colon \O_X^{\oplus m} \surj \cF.
\]
defined by these sections.
We denote by $\Gr(\C^m,r)$ the Grassmannian variety
of quotient maps $\C^m \surj \C^r$. Let $\p \colon X \times \Gr(\C^m,r) \to X$
be the projection onto the first factor, and consider the sequence
\[
0 \to \cS \to \p^*\O_X^{\oplus m} \to \cQ \to 0
\]
induced by the universal sequence of $\Gr(\C^m,r)$. We define
the Grassmannian bundle of rank $r$ locally free quotients of $\cF$ 
to be the subset
\[
\Gr(\cF,r) \subset X \times \Gr(\C^m,r)
\]
where the induced map $\p^*\ker(\f) \to \cQ$ is zero. 
This set is locally defined by finitely many holomorphic
sections of $\cQ$, and hence it is an analytic subset
of $X \times \Gr(\C^m,r)$.
Over the open and dense set $U \subset X$ where $\cF$ is a locally
free sheaf of rank $r$, the induced map $\Gr(\cF|_U,r) \to U$
is an isomorphism. Denote by $N(\cF)$ the irreducible component of $\Gr(\cF,r)$
containing $\Gr(\cF|_U,r)$. We define the Nash transformation of $\cF$
to be the induced map
\[
\n \colon N(\F) \to X.
\]
Note that $N(\F)$, being an irreducible component of 
a closed analytic subset, is a closed analytic subvariety of $X \times \Gr(\C^m,r)$, 
and $\n$ is a proper bimeromorphic morphism.

\begin{lemma}
\label{lem:universal-property}
The pull-back $\n^*\cF$ admits a locally free quotient of rank $r$, 
and $\n$ is universal 
among proper bimeromorphic morphisms of analytic varieties with this property. 
\end{lemma}

\begin{proof}
There is a commutative diagram of $\O_{N(\cF)}$-modules with exact rows as follows:
\[
\xymatrix{
0 \ar[r] & \cS|_{N(\cF)} \ar[r] 
& (\p^*\O_X^{\oplus m})|_{N(\cF)} \ar@{=}[d] \ar[r] & \cQ|_{N(\cF)} \ar[r] & 0 \\
0 \ar[r] & \n^*\ker(\f) \ar[r] \ar@{^(->}[u]^\iota 
& \n^*\O_X^{\oplus m} \ar[r] & \n^*\cF \ar@{->>}[u]_\t \ar[r] & 0
}.
\]
The sheaf $\cQ|_{N(\cF)}$ is locally free
of rank $r$, and the sheaf $\n^*\ker(\f)$ is torsion free and agrees with $\cS|_{N(\cF)}$ over
the open set $U$ where $\cF$ is locally free. The existence of the injection $\iota$ 
follows from the fact that the top row is locally split and $\n^*\ker(\f)$ is torsion free, 
and implies the existence of the surjection $\t$. 
The universal property follows by the universal property of the Grassmannian.
\end{proof}

This property allows us to remove the condition that $\cF$ is finitely generated
by global sections and globalize the construction of Nash transformation. 
Given any coherent sheaf $\cF$ on a complex analytic variety $X$, 
we find an open cover $\{U_i \subset X\}$ such that the restriction of
$\cF$ to each open $U_i$ is finitely generated by global sections. 
The transformations $N(\cF|_{U_i}) \to U_i$ are constructed as discussed above, 
and \cref{lem:universal-property} guarantees that these maps can be glued together to define
a morphism of analytic varieties $\n \colon N(\cF) \to X$, 
which we call the Nash transformation of $\cF$. Note that, by construction, 
isomorphic coherent sheaves have isomorphic Nash transformations. 
\cref{lem:universal-property} yields the following property. 

\begin{proposition}
\label{th:N(F)-analytic}
The Nash tranformation $\n \colon N(\cF) \to X$ of
a coherent sheaf $\F$ of rank $r$ 
on a complex analytic variety $X$ is a proper bimeromorphic morphisms
of analytic varieties, and is universal with respect to the property that 
$\n^*\cF$ has a locally free quotient of rank $r$.
\end{proposition}

We will use the following elementary property.

\begin{lemma}
\label{lem:N(F)=N(detF)=N(F/tor)}
Let $\F$ be a coherent sheaf of rank $r$ on a complex analytic variety $X$, 
and denote by $\t(\F)$ the torsion subsheaf of $\F$.
Then the Nash transformations of the sheaves $\F$, $\wedge^r \cF$, and $\F/\t(\F)$
are all isomorphic.
\end{lemma}

\begin{proof}
The identification $N(\F) = N(\F/\t(\F))$ is clear by construction, and
working locally so that we have a surjection $\O_X^{\oplus m} \surj \F$, 
one sees that $N(\F)$ is mapped isomorphically to $N(\wedge^r\F)$ via
the Pl\"ucker embedding $X \times \Gr(\C^m,r) \inj X \times \P(\wedge^r\C^m)$.
\end{proof}

\begin{remark}
The Nash transformation of a coherent sheaf $\F$ on a complex analytic variety $X$
is, locally, the blow-up of a finitely generated ideal. Indeed, 
\cref{lem:N(F)=N(detF)=N(F/tor)} implies that the Nash transformation of $\F$
is the same as the Nash transformation of $\wedge^r\F/\t(\wedge^r\F)$, and 
every point of $X$ has an open neighborhood over which
the latter is isomorphic to a finitely generated ideal. 
Realizing a Nash transformation as the blow-up of a finitely generated ideal
is the approach followed by Nobile \cite{Nob75} to prove that the Nash blowing-up
of a complex analytic variety (namely, the Nash transformation of $\Om_X^1$) 
is a holomorphic map of analytic varieties
(a different proof of this fact is given in \cite{Whi65}).
\end{remark}

Before we proceed with the proofs of \cref{th:Plateau-CY,th:Plateau-hypersurface}, 
we briefly discuss the Zariski--Lipman conjecture.
The conjecture (stated in \cite{Lip65} as a question)
is that any algebraic variety $X$ defined over a field of characteristic zero whose
the tangent sheaf $(\Om_X^1)^*$ is locally free is smooth. 
The same conjecture can be made for complex analytic varieties. 

One first evidence was obtained in \cite{vSS85}, 
where the property is shown to hold for isolated singularities of dimension $\ge 3$. 
The setting is analytic, but the same argument works for algebraic varieties. 
The proof uses extension properties of differentials, a trick 
that has been borrowed in \cite{GKKP11} to settle the conjecture for 
varieties with log terminal singularities. We will apply the result of \cite{vSS85}
in the proof of \cref{th:Plateau-CY}.

The other case that we will use in the proof of \cref{th:Plateau-hypersurface} is due to \cite{Kal11}, 
where the property is established for varieties with locally complete intersection singularities.

\begin{remark}
\label{rmk:Kal11-analytic}
The result of \cite{Kal11} is stated in the algebraic setting, but 
for isolated singularities it
can easily be extended to the analytic setting, as follows. 
Let $X$ be a complex analytic variety with a locally complete intersection isolated singularity at a point $x$.
By Artin's approximation theorem \cite[Theorem~3.8]{Art69}, 
there is a complex algebraic variety $Y$ with a closed point $y \in Y$ 
such that $X$ and $Y^\an$ have isomorphic open neighborhoods
at these points. It follows that 
$Y$ has a locally complete intersection isolated singularity at $y$
(e.g., see \cite[Tag~09Q6]{stacks-project}).
We have $(\Om_Y^1)^\an = \Om_{Y^\an}^1$, and hence 
$((\Om_Y^1)^*)^\an = (\Om_{Y^\an}^1)^*$ since taking duals of coherent $\O_Y$-modules commutes with 
analytification.
If $(\Om_X^1)^*$ is locally free at $x$ then $(\Om_Y^1)^*$ is locally free 
at $y$, and therefore $Y$ is smooth at $y$ by \cite[Corollary~1.3]{Kal11}. 
This can only occur if $X$ is smooth at $x$. 
\end{remark}

\begin{proof}[Proof of \cref{th:Plateau-CY}]
If $\~V$ is smooth then $g^{(\wedge^n1)}(M)=0$ by \cref{lem:g(M)}. 
Conversely, assume that $g^{(\wedge^n1)}(M)=0$. 
\cref{lem:g(M)} implies that the map $\wedge^n\Om_{\~V}^{[1]} \to \Om_{\~V}^{[n]}$
is surjective, and since the kernel of this map
is the torsion of $\wedge^n\Om_{\~V}^{[1]}$, we have
\[
\Om_{\~V}^{[n]} \cong \wedge^n\Om_{\~V}^{[1]}/\t(\wedge^n\Om_{\~V}^{[1]}).
\]
Note that $\Om_{\~V}^{[n]} \cong \O_{\~V}$ by \cref{lem:CY} and the assumption that $M$ is Calabi--Yau.
This implies that the Nash transformation of 
$\wedge^n\Om_{\~V}^{[1]}/\t(\wedge^n\Om_{\~V}^{[1]})$ is an isomorphism, and hence
it follows by \cref{lem:N(F)=N(detF)=N(F/tor)} that the Nash transformation of 
$\Om_{\~V}^{[1]}$ is an isomorphism.
By \cref{th:N(F)-analytic}, we see this can only occur if $\Om_{\~V}^{[1]}$ has a locally free quotient
of the same rank, 
and since $\Om_{\~V}^{[1]}$ is torsion free, this means that $\Om_{\~V}^{[1]}$ is locally 
free. 

Since $\~V$ is reduced, we have $(\Om_{\~V}^1)^* = (\Om_{\~V}^1)^{***}$, and therefore
$\Om_{\~V}^{[1]}$ is locally free if and only if $(\Om_{\~V}^1)^*$ is locally free.  
To finish the proof, we need to argue that $(\Om_{\~V}^1)^*$ can only be locally free
if $\~V$ is smooth. This is precisely the content of the Zariski--Lipman conjecture. 
As $\~V$ has isolated singularities and
we are assuming that $n \ge 3$, this is proved in \cite{vSS85}.
\end{proof}

\begin{corollary}
\label{th:Plateau-Gor}
Let $V \subset \C^N$ be a bounded Stein space of dimension $n \ge 3$ with isolated singularities and 
boundary regularity, and let $M$ denote the boundary. Let $\~V \to V$ be the normalization, and
assume that $\Om_{\~V}^{[n]}$ is a line bundle. 
Then $\~V$ is smooth if and only if $g^{(\wedge^n1)}(M) = 0$.
\end{corollary}

\begin{proof}
For each singular point $v_i \in V$, let $V_i \subset V$ to be the open neighborhood cut by 
an open ball of sufficiently small radius centered at $v_i$, and let $M_i$ denote the boundary of $V_i$.
If $\~V_i \to V_i$ is the normalization map, then we can ensure that 
$\Om_{\~V_i}^{[n]} \cong \O_{\~V_i}$ for every $i$. Note that the sets $\~V_i$ cover
the singular points of $\~V$.
By \cref{lem:g(M)}, we have
\[
g^{(\wedge^n1)}(M) = \sum g^{(\wedge^n1)}(M_i).
\]
and by \cref{th:Plateau-CY} we see that each $\~V_i$ is smooth if and only if $g^{(\wedge^n1)}(M_i) = 0$. 
Since the numbers $g^{(\wedge^n1)}(M_i)$ are nonnegative, this occurs for every $i$ if and 
only if $g^{(\wedge^n1)}(M) = 0$.
\end{proof}

We obtain \cref{th:Plateau-hypersurface} as a special case of the following result.

\begin{theorem}
\label{th:Plateau-lci}
Let $V \subset \C^N$ be a bounded Stein space of dimension $n \ge 2$ with isolated locally complete intersection
singularities and boundary regularity, and let $M$ denote the boundary.
Then $V$ is smooth if and only if $g^{(\wedge^n1)}(M) = 0$.
\end{theorem}

\begin{proof}
The space $V$ is normal and Gorenstein, and hence $\Om_V^{[n]}$ is a line bundle.
The same argument used in the proof of \cref{th:Plateau-CY} shows that $\Om_V^{[1]}$ is locally free, 
and we conclude this time by applying \cite{Kal11} (cf.\ \cref{rmk:Kal11-analytic}).
\end{proof}

\begin{proof}[Proof of \cref{th:Plateau-hypersurface}]
This follows directly from \cref{th:Plateau-lci} since
the filling $V$ is a hypersurface in $\C^{n+1}$.
\end{proof}

\begin{remark}
An alternative way of writing these proofs is to apply Artin's approximation theorem
early on in the proofs, thus reducing to work with Nash transformations and the Zariski--Lipman
conjecture in the algebraic setting. 
For clarity of exposition, 
we have opted to stick to the analytic setting, even though this requires a little preparation.
In this way we also avoid relying on Artin's theorem in the proof of \cref{th:Plateau-CY}.
\end{remark}

\section{Proof of \cref{th:link-invariant}}

In order to define a link-theoretic invariant of singularity using CR geometry, we
introduce a suitable equivalence relation among compact strongly pseudoconvex CR structures.
We start by recalling \cite[Definition~5.3]{dFT17}, which is in turn modeled upon
\cite[Definition~1.1]{HLY06}.

\begin{definition}
\label{d:CR-family}
A \emph{CR deformation family} of relative dimension $2n-1$ is a
strongly pseudoconvex CR manifold $\M$ of dimension $2n+1$ with a
proper $C^{\infty}$ submersion
$f\colon \M \rightarrow \D = \{t \in \C \mid |t| < 1\}$, such that for any $t \in \C$, the fiber
$\M_t = f^{-1}(t)$ is a strongly pseudoconvex CR manifold with CR structure
$T_{1,0}(\M_t) = T_{1,0}(\M) \cap \C T(\M_t)$.
\end{definition}

The next two definitions should be compared to \cite[Definitions~5.5 and~5.6]{dFT17}.

\begin{definition}
\label{d:cohomologically CR-family}
A CR deformation family $f\colon \M \rightarrow \D$ as in \cref{d:CR-family}
is said to be \emph{semirigid} if 
the invariant $g^{\wedge^n1}(\M_t)$ and, for $n \ge 3$,  
the dimension of the Kohn--Rossi cohomology group $H_\KR^{0,1}(\M_t)$
are constant as functions of $t \in \D$.
\end{definition}

\begin{definition}
\label{d:CR-equiv}
Two compact strongly pseudoconvex CR manifolds $(M, T_{1,0}(M))$ and
$(M', T_{1,0}(M'))$ of the same dimension $2n-1$ are 
said to be \emph{semirigid CR deformation equivalent}
if there exists a finite collection of semirigid CR deformation families
$f^i\colon \M^i \rightarrow \D$ of relative dimension $2n-1$, 
indexed by $i = 1,\dots,k$, and pairs of points $s_i,t_i \in \D$, such that
\begin{enumerate}
\item
$(\M^i_{s_i}, T_{1,0}(\M^i_{s_i})) \simeq (\M^{i+1}_{t_{i+1}}, T_{1,0}(\M^{i+1}_{t_{i+1}}))$ 
for $i < k$;
\item
$(\M^1_{s_1}, T_{1,0}(\M^1_{s_1})) \simeq (M, T_{1,0}(M))$ and 
$(\M^k_{t_k}, T_{1,0}(\M^k_{t_k})) \simeq (M', T_{1,0}(M'))$.
\end{enumerate}
\end{definition}

It is immediate to see that this definition yields an equivalence relation
among compact strongly pseudoconvex CR manifolds.

Let $X \subset \C^N$ be an $n$-dimensional complex analytic variety
with a normal isolated singularity at the origin $0 \in \C^N$.
Consider the function $\r \colon X \cup \de X \to \R$ given by $z \mapsto |z|$.
A number $\a \in (0,\infty)$ is said to be a critical value of $\r$ if
$\a = \r(z)$ where either $z \in X_\reg$ is a point where $\r$ is not a submersion,
or $z \in (\Sing X) \cup \de X$. Let $\a^*(X) \in (0,\infty]$ be the supremum of the numbers $\a$
such that there are no a critical values of $\r$ in the interval $(0,\a)$. 
Then for every $\ep \in (0,\a^*(X))$ the set
\[
L_{X,\ep} := S^{2N-1}_\ep \cap X
\]
is a link of $X$. We regard $L_{X,\ep}$ as a CR manifold with the CR structure 
inherited from the embedding in $X$. 

The exact same proof of \cite[Theorem~5.7]{dFT17} yields the following result.

\begin{proposition}
\label{t:CR}
Suppose that $X \subset \C^N$ and $X' \subset \C^{N'}$ are two complex analytic
varieties with isomorphic germs of normal isolated singularities $(X,0) \cong (X',0')$
at the respective origins $0 \in \C^N$ and $0' \in \C^{N'}$.
Then for any $0 < \ep < \a^*(X)$ and $0 < \ep' < \a^*(X')$
the links $L_{X,\ep}$ and $L_{X',\ep'}$ are semirigid CR deformation equivalent.
\end{proposition}

The key fact in the proof of this property is that, for an embeddable strongly pseudoconvex 
compact CR manifold $M$, \cref{lem:g(M)} implies the value of $g^{\wedge^n1}(M)$ 
only depends on the germs of the normal filling of $M$ at its singular points,
and if $n \ge 3$, then so does the dimension of the group $H_\KR^{0,1}(M)$
by \cite[Theorem~B]{Yau81}. 

The following is a more precise restatement of \cref{th:link-invariant}.

\begin{theorem}
\label{th:CR-link-invariant}
An $n$-dimensional complex analytic variety $X \subset \C^N$ with only normal isolated singularities
is smooth at a point $0 \in X$ if and only if
the semirigid CR deformation equivalence class of the link of $X$ at $0$ is the class of the
unit sphere $S^{2n-1} \subset \C^n$ equipped with the standard CR structure.
\end{theorem}

\begin{proof}
Since we already know that links characterize smooth points among normal surface singularities \cite{Mum61}, 
we can assume that $n \ge 3$.
Suppose, by way of contradiction, that $X$ is actually singular
at $0$ but its link is semirigid CR deformation equivalent to the standard CR sphere. 
Let $f^i \colon \M^i \to \D$, for $i = 1,\dots,k$, 
be a sequence of semirigid CR deformation families as in Definition~\ref{d:CR-equiv}, 
connecting the standard CR sphere $S^{2n-1}$ to a link of $0 \in X$ with its CR structure. 
Note that $g^{\wedge^n1}(\M^i_t) = H_\KR^{0,1}(\M^i_t) = 0$ for some $i$ and some $t$, and therefore for
all $i$ and $t$.

By \cite[Main Theorem]{HLY06}, for every $i$ 
there exist a Stein space $\U^i$ which has $\M^i$ as part of its smooth boundary, 
and a holomophic map $\U^i \to \D$ such that, for every $t \in \D$, 
the fiber $\U^i_t$ is a Stein space with boundary regularity and has $\M^i_t$ as its boundary. 
Since $H_\KR^{0,1}(\M^i_t) = 0$ for all $i$ and $t$,
it follows by \cite[Corollary~1.5]{HLY06} that each fiber $\U^i_t$ has isolated normal singularities.
Using the fact that normal Stein fillings are unique up to isomorphism
(this property is implicit in \cite{Ros65}, see \cite[Corollary~4.3]{dFT17}),
we see that there are isomorphisms $\U^i_{t_i} \cong \U^{i+1}_{s_{i+1}}$ for all $i < k$. 
Note that while $\U^1 \to \D$ has a smooth and contractible fiber, 
one of the fibers of $\U^k \to \D$ is singular, isomorphic to an open neighborhood of $0$ in $X$.

By applying Ehresmann's fibration theorem in the context of families of manifolds with boundary
when moving across a family with smooth fibers, 
and using the isomorphisms $\U^i_{t_i} \cong \U^{i+1}_{s_{i+1}}$ 
to move from one family to the next,
we end up with a family $\U^{i_0} \to \D$, for some index $i_0$,
containing both a singular fiber and a smooth and contractible one. 
For simplicity, we drop the index $i_0$ and denote this family by $\U$
and its relative boundary over $\D$ by $\M$.
After possibly shrinking and re-parameterizing the base, we can assume without loss of generality
that every fiber $\U_t$, for $t \ne 0$, is smooth and contractible and the 
central fiber $\U_0$ is singular with normal isolated singularities. 
Note that $\U$ has normal singularities. 

By adjunction, we have $\Om_{\U_t}^{[n]} = (\Om_\U^{[n+1]}|_{\U_t})^{**}$ for every $t$, 
and since $\U \to \D$ is smooth away from the central fiber, 
we actually have $\Om_{\U_t}^{[n]} = \Om_\U^{[n+1]}|_{\U_t}$ for $t \ne 0$.
By Cartan's theorem B and the contractibility of the smooth fibers,
we have $\Pic(\U_t) = 0$, and hence $\Om_{\U_t}^{[n]} \cong \O_{\U_t}$,
for every $t \ne 0$. This implies that $\Om_\U^{[n+1]} \cong \O_\U(D)$
for some divisor $D$ supported within the central fiber $\U_0$. 
Since $\U_0$ is irreducible and reduced, 
and linearly equivalent to zero as a divisor on $\U$, it follows that $\Om_\U^{[n+1]} \cong \O_\U$. 
We conclude that $\Om_{\U_0}^{[n]} \cong \O_{\U_0}$. 
By \cref{lem:CY}, this means that the boundary $\M_0$ is a Calabi--Yau CR manifold. 
As we have $g^{\wedge^n1}(\M_0) = 0$, \cref{th:Plateau-CY} implies that $\U_0$ is smooth, 
in contradiction with our starting assumption.
\end{proof}

\end{document}